\documentclass[10pt,reqno]{amsart}
\usepackage[a4paper,left=20mm,right=20mm,top=30mm,bottom=30mm,marginpar=25mm]{geometry}

\usepackage{amsfonts}
\usepackage{amsthm}
\usepackage{amssymb}
\usepackage{amsmath}
\usepackage{amscd}
\usepackage{mathtools}
\usepackage{color}
\usepackage{enumerate}
\usepackage{dsfont}
\usepackage[latin1]{inputenc}
\usepackage[mathscr]{eucal}
\numberwithin{equation}{section}
\usepackage{hyperref}
\usepackage{indentfirst}
 
 \makeindex

\theoremstyle{plain}
\newtheorem{theorem}{Theorem}[section]
\newtheorem{lemma}[theorem]{Lemma}

\newtheorem{proposition}[theorem]{Proposition}

\theoremstyle{definition}

\newtheorem{remark}[theorem]{Remark}

\title[Asymptotic limits in bipolar Euler-Poisson systems]
      {Zero-electron-mass and quasi-neutral limits \\ for bipolar Euler-Poisson systems}

\author{Nuno J. Alves}
\address[Nuno J. Alves]{
        Faculty of Mathematics, University of Vienna, Oskar-Morgenstern Platz 1, 1090 Vienna, Austria.}
\email{nuno.januario.alves@univie.ac.at}
\author{Athanasios E. Tzavaras}
\address[Athanasios E. Tzavaras]{CEMSE Division, King Abdullah University of Science and Technology, Thuwal, Saudi Arabia, 23955-6900.}
\email{athanasios.tzavaras@kaust.edu.sa}

\subjclass{Primary: 35Q31, 35Q35; Secondary: 35L65, 76W05.}
 \keywords{Bipolar Euler-Poisson, zero-electron-mass limit, quasi-neutral limit, relative energy, Neumann function, Riesz potentials}

\begin{document}

\begin{abstract}
We consider a set of bipolar Euler-Poisson equations and study two asymptotic limiting processes. The first is the zero-electron-mass limit, which
formally results in a non-linear adiabatic electron system. In a second step, we analyse the combined zero-electron-mass and quasi-neutral limits, 
which together lead to the compressible Euler equations. Using the relative energy method, we rigorously justify these limiting processes
for weak solutions of the two-species Euler-Poisson equations that dissipate energy, as well as for strong solutions of the limit systems that are bounded away from vacuum. This justification is valid in the regime of initial data for which strong solutions exist. To deal with the electric potential, in the first case we use elliptic theory, whereas in the second case we employ the theory of Riesz potentials and properties of the Neumann function.
\end{abstract}

\maketitle

\baselineskip 16pt

\section{Introduction}
Consider the following bipolar Euler-Poisson equations:
\begin{equation} \label{BEP}
 \begin{split}
  \partial_t \rho + \nabla \cdot (\rho u) & =  0, \\
  \rho  \partial_t  u + \rho (u \cdot \nabla)  u  & =   - \nabla P_1(\rho) - \rho \nabla \phi, \\
  \partial_t n + \nabla \cdot (n v) & =   0, \\
  \varepsilon \big( n  \partial_t v + n(v \cdot \nabla)v  \big) & =  - \nabla P_2(n) +  n \nabla \phi, \\
  -\delta \Delta \phi & =   \rho - n, \ 
 \end{split}
\end{equation} 
in $\mathopen{]}0,T\mathclose{[} \times \Omega,$ together with the boundary conditions 
\begin{equation*}
u \cdot \nu = v \cdot \nu = \nabla \phi \cdot \nu = 0, 
\end{equation*}
on $\mathopen{[}0,T\mathclose{[} \times \partial \Omega$, where $T>0,$ $\Omega \subseteq \mathbb{R}^d$ is a smooth bounded domain, $d \in \mathbb{N}$ is the spatial dimension, and $\nu$ is an arbitrary outer normal vector to the boundary $\partial \Omega$.
These equations describe the evolution of a two-species fluid composed of positively and negatively charged particles subject to a self-created electric field. It is a basic model which together with its various asymptotic limits is widely used in semiconductor theory or plasma physics \cite{chen1984introduction, freidberg2008plasma, markowich2012semi}. For theoretical results concerning existence of global solutions near equilibrium for the two-fluid Euler-Poisson model we refer to \cite{guo2011global, guo2014global, guo2016global}. 
It was there observed that the electrical interaction can induce strong dispersive effects, enhance linear decay rates, and thus prevent shock formation.  

System (\ref{BEP}) is composed of two compressible Euler systems, with $\rho,n$ denoting the densities, $u,v$ the respective linear velocities, and  $P_1,P_2$ the associated pressures. 
One may regard $\rho$ as the density of positively charged ions, and $n$ as the density of electrons.
On the right-hand-side of the momentum equations one finds the electric part of the Lorenz force.
The Euler equations are coupled via a Poisson equation determining the electric potential $\phi$ and the associated electric field $E = - \nabla \phi$.
The constants $\varepsilon, \delta > 0$ represent the ratio of electron mass to ion mass and the scaled Debye length squared, respectively. 
In plasma physics, the Debye length stands for the scale over which mobile charge carriers screen out electric fields in plasmas and other conductors. 
One may view the Debye length as a length-scale over which significant charge separation occurs \cite{chen1984introduction,freidberg2008plasma}.
One commonly employs various simplified systems derived from the bipolar Euler-Poisson and obtained through various limiting processes.
A class of simplifications occurs when friction is added to \eqref{BEP}; one considers the high-friction limit which leads to the bipolar drift-diffusion equations commonly used in semiconductor modelling \cite{markowich2012semi}. This limit has been  studied extensively, 
{\it e.g.}  \cite{lattanzio1999relaxation,lattanzio2000bipolar,alves2022relaxation} and references therein.

Other limiting processes that lead to simplified systems consist of the zero-electron-mass limit and the quasi-neutral limit. 
These limits are relevant for fast flows which occur in plasmas and are the subjects of study in this work. The zero-electron-mass limit refers to the limit of the two-species Euler-Poisson system (\ref{BEP}) as $\varepsilon \to 0,$ while the quasi-neutral limit is understood to be the limit of the system (\ref{BEP}) as $\delta \to 0.$ 
When  $\varepsilon \to 0$,  the second momentum equation in (\ref{BEP}) formally reduces to $\nabla P_2(n) =  n \nabla \phi$. Accordingly,
 system (\ref{BEP}) reduces to 
\begin{equation} \label{AE}
 \begin{split}
  \partial_t \rho + \nabla \cdot (\rho u) & = 0, \\
  \rho  \partial_t  u + \rho (u \cdot \nabla)  u  & = - \nabla P_1(\rho) - \rho \nabla H_2^\prime(n),  \\
  -\delta \Delta H_2^\prime(n) + n & = \rho, 
  \end{split}
\end{equation}
while the boundary conditions become
\begin{equation*}
u \cdot \nu  = \nabla H_2^\prime(n) \cdot \nu = 0, 
\end{equation*}
where $H_2 = H_2(n)$ satisfies  $P_2^\prime(n) = n H_2^{\prime \prime}(n).$ Going back to the relation $\nabla P_2(n) =  n \nabla \phi$ and 
considering $P_2(n) = n$,  one obtains $n = n_0 e^{ \phi},$
which is the case for an adiabatic electron system \cite{horton1999drift, jorge2018theory}, and is also known as the Boltzmann relation \cite{guo2011global}, where $n_0$ is the initial data. Thus, (\ref{AE}) can be regarded as a non-linear version of the standard adiabatic electron system. To the best of the authors' knowledge, an existence theory for (\ref{AE}) is not yet available in the literature, therefore representing a subject for further investigation. \par
Furthermore, in the combined limit $(\varepsilon,\delta) \to (0,0)$ one formally obtains the following Euler equations 
\begin{equation} \label{E}
 \begin{split}
  \partial_t \rho + \nabla \cdot (\rho u) & = 0, \\
  \rho  \partial_t  u + \rho (u \cdot \nabla)  u  & = - \nabla\big( P_1(\rho) + P_2(\rho) \big), 
  \end{split}
\end{equation}
with the boundary condition $u \cdot \nu = 0$. Regarding existence theories for Euler systems we refer to \cite{schochet1986compressible, hou2017global}. \par 
A rigorous justification of the previous limiting processes is the intent of this work. Both limits have been investigated in some contexts. In the case where the limit system is a incompressible Euler system, we refer to \cite{loeper2005quasi} for the quasi-neutral limit of a single-species Euler-Poisson system in a periodic setting, and to \cite{xu2013zero} for the zero-electron-mass limit of a two-species Euler-Poisson system in several space dimensions. Regarding the case where the limit system is a model for compressible fluids, we refer to \cite{peng2022global} and to \cite{ju2019quasineutral} for the quasi-neutral limit of a two-species Euler-Poisson system towards a single-species Euler system, in one and three spatial dimensions, respectively.
In all of the aforementioned works, the analysis is conducted in a framework of smooth solutions with initial data close to an equilibrium state. \par
In the sequel, we consider weak solutions for the bipolar Euler-Poisson equations (\ref{BEP}), which are reasonable in light of the a priori energy estimates of the system. The relative energy method is used to analyse the two limits of (\ref{BEP}), first towards (\ref{AE}) when $\varepsilon \to 0,$ and then towards (\ref{E}) as $\varepsilon,  \delta \to 0$, in both cases for strong solutions of the limiting systems. This method provides an efficient framework for limiting processes and stability analysis \cite{tzavaras2005relative, lattanzio2013relative}, {\it e.g.},  for  the relaxation limit of Euler and Euler-Poisson systems \cite{lattanzio2017gas, carrillo2020relative, alves2022relaxation}. The present work thus extends the applicability of the relative energy method in order to include other types of asymptotic limits relevant in plasma physics but not necessarily related to relaxation. The zero-electron-mass limit is considered when the mass of ions is much heavier than the mass of electrons, whereas the quasi-neutral limit refers to a plasma in which the Debye length is negligible in comparison to the characteristic length of the system. The Debye length defines the length scale on which the electric effects created by the particles in the plasma are detected. At length scales bigger than the Debye length one observes an overall electric neutrality. \par 
The bipolar Euler-Poisson equations (\ref{BEP}) are formally analysed in Section \ref{sectionBEP}. One starts by deriving (\ref{BEP}) using non-dimensional analysis. Secondly, one formally obtains the energy identity for the system and conducts an asymptotic expansion. The local form of the relative energy identity for the system is also given, and the section is completed with the notion of weak solutions for (\ref{BEP}) which will be used in the subsequent analysis. Then, one proceeds with the relative energy method to establish the limits of (\ref{BEP}) towards (\ref{AE}) and (\ref{E}), this being the content of Section \ref{zem} and Section \ref{zemqn}, respectively. This approach relies on an inequality satisfied by relative energy for the bipolar Euler-Poisson equations, where solutions of (\ref{BEP}) are either compared with solutions of (\ref{AE}) or of (\ref{E}); see Theorem \ref{AEmaintheorem} and Theorem \ref{maintheorem}. In order to perform such a comparison, one regards the solutions of (\ref{AE}) and (\ref{E}) as approximate solutions of (\ref{BEP}); see (\ref{approxAE}) and (\ref{approxE}).

\section{The bipolar Euler-Poisson system} \label{sectionBEP}
\subsection{Nondimensionalization}
As the starting point of this analysis, consider the following bipolar Euler-Poisson equations with dimensional constituents:
\begin{equation} \label{dimBEP}
\begin{split}
\partial_t n_i + \nabla \cdot (n_iu_i) & = 0, \\
\partial_t (m_i n_i u_i) + \nabla \cdot (m_i n_i u_i \otimes u_i) & = - \nabla P_i - q_i n_i \nabla \phi,\\
\partial_t n_e + \nabla \cdot (n_eu_e) & = 0, \\
\partial_t (m_e n_e u_e) + \nabla \cdot (m_e n_e u_e \otimes u_e) & = - \nabla p_e - q_e n_e \nabla \phi, \\
-\varepsilon_0 \Delta \phi & = q_i n_i + q_e n_e. \\
\end{split}
\end{equation}
The equations above describe the evolution of an unmagnetized plasma composed of two-species of charged particles, {\it e.g.}, ions and electrons. In what follows, the S.I. units of the quantities are placed in between parenthesis . For a species $j = i,e$, the density is denoted by $n_j  \ (m^{-3}),$ the linear velocity is represented by $u_j \ (m \cdot s^{-1}),$ $p_j \ (kg \cdot m^{-1} \cdot s^{-2})$ is the pressure, $m_j \ (kg)$ is the mass and $q_j \ (s \cdot A)$ is the charge. The electric potential created by the charged particles is denoted by $\phi \ (kg \cdot m^2  \cdot s^{-3} \cdot A^{-1}) .$  The permittivity and permeability of free space are denoted by $\varepsilon_0 \ (kg^{-1} \cdot m^{-3} \cdot s^4 \cdot A^2)$ and $\mu_0 \ (kg \cdot m \cdot s^{-2} \cdot A^{-2} ),$ respectively. It holds that $\varepsilon_0 \mu_0 = 1/c^2,$ with $c$ being the speed of light.  \par 
Let $e$ be the elementary charge and assume for simplicity that $q_i = - q_e = e.$ This is the case of plasmas that consist in hydrogen, deuterium, or a mix of deuterium and tritium \cite{freidberg2008plasma}. Consider the scaling:
\[x=L x^\prime  , \quad t = \tau_0 t^\prime  , \]
\[n_j = N_0 n_j^\prime  , \quad u_j = v_0 u_j^\prime  , \quad p_j = \kappa_B T_0 N_0 p_j^\prime  , \quad j=i,e  , \]
\[\phi = \frac{\kappa_B T_0}{e} \phi^\prime  , \]
where $v_0=L/\tau_0,$ $\kappa_B$ is the Boltzmann constant and $T_0$ is the temperature of the system. \par 
After dropping the primes, setting $\rho = n_i,$ $n = n_e,$ $u = u_i,$ $v=u_e, $ $P_i = P_1, $ $p_e = P_2,$ and assuming that the pressures depend solely on the densities, one obtains the non-dimensional version of (\ref{dimBEP}):
\begin{equation*} 
\begin{split}
\partial_t \rho + \nabla \cdot (\rho u) & = 0, \\
\zeta \big( \rho \partial_t  u + \rho (u \cdot \nabla )u  \big)  & = - \nabla P_1(\rho) - \rho \nabla \phi,\\
\partial_t n + \nabla \cdot (n v) & = 0, \\
\varepsilon \big( n \partial_t  v + n (v \cdot \nabla )v \big)  & = - \nabla P_2(n) + n \nabla \phi,\\
- \delta \Delta \phi & = \rho - n, 
\end{split}
\end{equation*}
where \[\zeta = \frac{m_i v_0^2}{\kappa_B T_0} ,  \quad \varepsilon =  \frac{m_e v_0^2}{\kappa_B T_0}  , \quad \delta = \frac{\lambda_D^2}{L^2} = \frac{\varepsilon_0 \kappa_B T_0}{e^2 L^2 N_0}  , \]
with $\lambda_D$ being the Debye length. \par
Note that $\zeta = 1$ is equivalent to $\varepsilon = m_e/m_i,$ the ratio of electron mass to ion mass, yielding system (\ref{BEP}).

\subsection{Pressure and internal energy}
One introduces the internal energy functions $H_1, H_2$ which are related with the pressures $P_1,P_2$ via
\begin{equation} \label{thermorelations}
rH_i^{\prime \prime}(r) = P_i^\prime(r)  , 
\end{equation}
for $r>0$ and $i=1,2.$
Integrating the identity above yields that
\begin{equation*}
rH_i^{\prime}(r) = H_i(r) + P_i(r)   .
\end{equation*} \par 
The following assumptions are placed on the pressure functions:
 \[ P_i \in C^2(]0, \infty[) \cap C([0,\infty[)  , \]   \[ P_i^\prime(r) > 0  , \quad 
|r P_i^{\prime \prime}(r)| \leq \hat{k}_i P_i^\prime(r)  , 
\]
for $r>0.$ As for the internal energy functions one assumes that
\[ H_i \in C^3(]0,\infty[)\cap C([0,\infty[)  , \] 
\[
\lim\limits_{r \to \infty} r^{-\gamma_i} H_i(r) = k_i(\gamma_i -1)^{-1}  ,
\]
for some $\gamma_i > 1$ and $k_i,\hat{k}_i > 0, \ i=1,2$.
As an example of a pair of functions satisfying the above conditions consider
\[p(r) =  r^\gamma , \quad h(r)= (\gamma -1)^{-1}r^\gamma , \quad \gamma > 1  . \]   
\subsection{Energy identity and asymptotics}
In this section one derives, by exact formal computation, the energy identity for (\ref{BEP}) and produces an asymptotic analysis to make a preliminary justification of the appearance of (\ref{E}) from (\ref{BEP}) in the combined limit.  \par 
Let $(\rho, u, n , v, \phi)$ solve (\ref{BEP}). Taking the inner product of the second and fourth equations of (\ref{BEP}) with $u$ and $v,$ respectively, and using the continuity equations yields
\begin{equation*}
\partial_t \big(\tfrac{1}{2} \rho |u|^2 + H_1(\rho) \big) + \nabla \cdot \big(\tfrac{1}{2} \rho |u|^2u + \rho H_1^\prime(\rho)u \big) +\rho u \cdot \nabla \phi = 0 ,
\end{equation*}
and
\begin{equation*}
\partial_t \big(\varepsilon \tfrac{1}{2} n |v|^2 + H_2(n) \big) + \nabla \cdot \big(\varepsilon\tfrac{1}{2} n |v|^2v + n H_2^\prime(n)v \big) - nv \cdot \nabla \phi = 0 . 
\end{equation*}
Adding together the two previous equations and using the Poisson equation results in the local form of the energy identity of (\ref{BEP}):
\begin{equation} \label{energylocalformBEP}
\partial_t \big( \tfrac{1}{2} \rho |u|^2 + H_1(\rho) + \varepsilon \tfrac{1}{2} n |v|^2 + H_2(n) + \delta \tfrac{1}{2} |\nabla \phi|^2 \big) + \nabla \cdot \mathcal{F}_1 = 0  , 
\end{equation}
where 
\[\mathcal{F}_1 = \mathcal{F}_1(\rho, u, n, v ,\phi) = \tfrac{1}{2} \rho |u|^2u + \varepsilon\tfrac{1}{2} n |v|^2v +\rho H_1^\prime(\rho)u +  n H_2^\prime(n)v + \phi(\rho u - nv) - \delta \phi \nabla \partial_t \phi \]
is the flux of the energy. Taking $\varepsilon = 0$ and $\phi = H_2^\prime(n)$ in (\ref{energylocalformBEP}) yields the local form of the energy identity for (\ref{AE}), 
\begin{equation} \label{energylocalformAE}
\partial_t \big(\tfrac{1}{2} \rho |u|^2 + H_1(\rho) + H_2(n) + \delta \tfrac{1}{2} |\nabla H_2^\prime (n)|^2 \big)  + \nabla \cdot \mathcal{F}_{2} = 0  ,
\end{equation}
where 
\[ \mathcal{F}_{2} = \mathcal{F}_{2}(\rho, u , n) = \tfrac{1}{2} \rho |u|^2u  +\rho H_1^\prime(\rho)u +  \rho H_2^\prime(n)u - \delta H_2^\prime (n) \nabla \partial_t H_2^\prime (n) \]
is the flux.
Furthermore, taking $\varepsilon = \delta = 0,$ $\rho = n$ and $u=v$ in (\ref{energylocalformBEP}) gives 
\begin{equation*} \label{energylocalformE}
\partial_t \big(\tfrac{1}{2} \rho |u|^2 + H_1(\rho) + H_2(\rho) \big)  + \nabla \cdot \mathcal{F}_3 = 0 ,
\end{equation*}
which is the local form of the energy identity for system (\ref{E}), with \[\mathcal{F}_3 = \mathcal{F}_3(\rho, u) = \tfrac{1}{2} \rho |u|^2u  +\rho H_1^\prime(\rho)u +  \rho H_2^\prime(\rho)u  . \] 
\par For solutions $(\rho, u, n, v, \phi),$ $(\rho, u ,n)$ and $(\rho, u)$ of (\ref{BEP}), (\ref{AE}) and (\ref{E}), respectively, the respective total energies $\mathcal{H}_0$, $\mathcal{H}_*$ and $\mathcal{H}_E$ are defined as follows:
\begin{equation} \label{energyBEP}
\mathcal{H}_0 = \mathcal{H}_0(\rho, u, n, v, \phi) =  \tfrac{1}{2} \rho |u|^2 + H_1(\rho) + \varepsilon \tfrac{1}{2} n |v|^2 + H_2(n) + \delta \tfrac{1}{2} |\nabla \phi|^2 , 
\end{equation} 
\begin{equation} \label{energyAE}
\mathcal{H}_* = \mathcal{H}_*(\rho, u, n) = \tfrac{1}{2} \rho |u|^2 + H_1(\rho) + H_2(n) + \delta \tfrac{1}{2} |\nabla H_2^\prime (n)|^2  ,
\end{equation} 
and
\begin{equation} \label{energyE}
\mathcal{H}_E = \mathcal{H}_E(\rho, u) = \tfrac{1}{2} \rho |u|^2 + H_1(\rho) + H_2(\rho)  .
\end{equation} 
Now, one performs an asymptotic analysis to the bipolar Euler-Poisson equations. Set \[m = \rho u  , \quad \mu = nv  , \] 
rewrite (\ref{BEP}) as follows
\begin{equation} \label{BEPm}
 \begin{aligned}
  \partial_t \rho + \nabla \cdot m = 0, \\
  \partial_t m + \nabla \cdot \Big(\frac{m \otimes m}{\rho}\Big) + \nabla P_1(\rho) + \rho \nabla \phi = 0, \\
  \partial_t n + \nabla \cdot \mu = 0, \\
  \varepsilon \Big(\partial_t \mu + \nabla \cdot \Big(\frac{\mu \otimes \mu}{n}\Big) \Big)+ \nabla P_2(n) -  n \nabla \phi = 0,\\
  -\delta \Delta \phi = \rho - n, 
 \end{aligned}
\end{equation} 
and introduce the asymptotic expansions in $\varepsilon:$
\begin{equation*} \label{expansion}
\begin{split}
\rho & = \rho_{0} + \varepsilon \rho_{1} + \varepsilon^2 \rho_{2} + \dots  , \\
m & = m_0 + \varepsilon m_1 + \varepsilon^2 m_2 + \dots  ,\\
n & = n_0 + \varepsilon n_1 + \varepsilon^2 n_2 + \dots  , \\
\mu & = \mu_0 + \varepsilon \mu_1 + \varepsilon^2 \mu_2 + \dots  , \\
\phi & = \phi_0 + \varepsilon \phi_1 + \varepsilon^2 \phi_2 + \dots  .\\
\end{split}
\end{equation*}
Expanding each one of the terms of (\ref{BEPm}), after gathering the $\mathcal{O}(1)$ terms it holds that 
\begin{equation*} 
 \begin{aligned}
  \partial_t \rho_0 + \nabla \cdot m_0 = 0, \\
  \partial_t m_0 + \nabla \cdot \Big(\frac{m_0 \otimes m_0}{\rho_0}\Big) + \nabla (\rho_0) + \rho_0 \nabla \phi_0 = 0, \\
  \partial_t n_0 + \nabla \cdot \mu_0 = 0, \\
 \nabla P_2(n_0) =  n_0 \nabla \phi_0, \\
  -\delta \Delta \phi_0 = \rho_0 - n_0,  
 \end{aligned}
\end{equation*} 
which can be solved by first obtaining $\rho_0, m_0, n_0$ from 
\begin{equation} \label{BEP0}
 \begin{aligned}
  \partial_t \rho_0 + \nabla \cdot m_0 = 0, \\
  \partial_t m_0 + \nabla \cdot \Big(\frac{m_0 \otimes m_0}{\rho_0}\Big) + \nabla P_1(\rho_0) + \rho_0 \nabla H_2^\prime(n_0) = 0,  \\
  -\delta \Delta H_2^\prime(n_0) + n_0 = \rho_0, 
  \end{aligned}
\end{equation}
and then setting 
\[\phi_0 = H_2^\prime(n_0) , \quad \mu_0 = m_0 - \delta \nabla \partial_t \phi_0 .\]
\par  
Furthermore, expanding each one of the terms in the energy identity (\ref{energylocalformBEP}) one obtains that $(\rho_0, m_0, n_0, \mu_0, \phi_0)$ satisfy the following energy identity:
\begin{equation} \label{energylocalformBEP0}
\partial_t \Big( \frac{1}{2} \frac{|m_0|^2}{\rho_0} + H_1(\rho_0) + H_2(n_0) + \delta \frac{1}{2}|\nabla \phi_0|^2 \Big) + \nabla \cdot \mathcal{F}_0 = 0  ,
\end{equation}
where 
\[\mathcal{F}_0  =  \frac{1}{2} \frac{|m_0|^2}{\rho_0^2}m_0 + H_1^\prime(\rho_0)m_0 + H_2^\prime(n_0) \mu_0 + \phi_0 (m_0- \mu_0) - \delta \phi_0 \nabla \partial_t \phi_0  . \]
Next, one makes a subsidiary expansion in $\delta$ of $(\rho_0, m_0, n_0),$
\begin{equation*} \label{expansion2}
\begin{split}
\rho_0 & = \rho_{00} + \delta \rho_{01} + \delta^2 \rho_{02} + \dots  , \\
m_0 & = m_{00} + \delta m_{01} + \delta^2 m_{02} + \dots  ,\\
n_0 & = n_{00} + \delta n_{01} + \delta^2 n_{02} + \dots  , \\
\end{split}
\end{equation*}
and plugs it in (\ref{BEP0}). Gathering the leading order terms one obtains that they satisfy 
\begin{equation*} \label{BEP00}
 \begin{aligned}
  \partial_t \rho_{00} + \nabla \cdot m_{00} = 0, \\
  \partial_t m_{00} + \nabla \cdot \Big(\frac{m_{00} \otimes m_{00}}{\rho_{00}}\Big) + \nabla P_1(\rho_{00}) + \rho_{00} \nabla H_2^\prime(n_{00}) = 0,  \\
  n_{00} = \rho_{00}, 
  \end{aligned}
\end{equation*}
which is equivalent to (\ref{E}) by (\ref{thermorelations}). Finally, one expands the terms in the energy identity (\ref{energylocalformBEP0}) and gathers the leading order terms to obtain 
\begin{equation} \label{energylocalformBEP00}
\partial_t \Big( \frac{1}{2} \frac{|m_{00}|^2}{\rho_0} + H_1(\rho_{00}) + H_2(n_{00}) \Big) + \nabla \cdot \mathcal{F}_{00} = 0 ,
\end{equation}
where 
\[\mathcal{F}_{00}  =  \frac{1}{2} \frac{|m_{00}|^2}{\rho_0^2}m_{00} + H_1^\prime(\rho_{00})m_{00} + H_2^\prime(n_{00}) \mu_{00} + \phi_{00} (m_{00}- \mu_{00}) . \]
Since $\mu_{00} = m_{00},$ one has that identity (\ref{energylocalformBEP00}) is in fact the local form of the energy identity for system (\ref{E}).
\subsection{Relative energy identity}
This section presents the relative energy identity for solutions of (\ref{BEP}). The derivation is conducted via an exact formal calculation, similar with the formal computations in \cite{alves2022relaxation}. Given a pair of smooth solutions, $(\rho,  u, n , v, \phi)$ and $(\bar \rho,  \bar u, \bar n,  \bar v, \bar \phi),$ of (\ref{BEP}), one uses the relative energy to compare them. In this case, the relative energy $\hat{\mathcal{H}}$ takes the form
\begin{equation} \label{relativeenergy}
\begin{split}
\hat{\mathcal{H}} & = \hat{\mathcal{H}}(\rho, u , n , v, \phi | \bar \rho,  \bar u, \bar n,  \bar v, \bar \phi) \\
& = \tfrac{1}{2} \rho |u-\bar u|^2 + H_1(\rho | \bar \rho) +\varepsilon \tfrac{1}{2} n |v - \bar v|^2 + H_2(n | \bar n) + \delta \tfrac{1}{2} |\nabla (\phi - \bar \phi)|^2  ,
\end{split}
\end{equation}
where for a function $F=F(\rho)$, $F(\rho | \bar \rho)$ is defined as
\begin{equation*} 
F(\rho | \bar \rho) = F(\rho) - F(\bar \rho) - F^\prime(\bar \rho)(\rho - \bar \rho)  .
\end{equation*} 
Taking the time derivative of $\hat{\mathcal{H}}$ renders the local form of relative energy identity for (\ref{BEP}):
\begin{equation} \label{localformrelenergyidentity}
\partial_t \hat{\mathcal{H}} + \nabla \cdot \hat{\mathcal{F}} = \hat{\Sigma}_1 + \hat{\Sigma}_2 + \hat{\Sigma}_*  ,
\end{equation}
where the relative flux $\hat{\mathcal{F}}$ is given by 
\[
\begin{split}
\hat{\mathcal{F}}  = \ &   \tfrac{1}{2} \rho |u - \bar u|^2u +\rho \big( H_1^\prime(\rho) - H_1^\prime(\bar \rho) \big)(u-\bar u) + H_1(\rho | \bar \rho) \bar u \\
& + \varepsilon\tfrac{1}{2} n |v - \bar v |^2v +  n \big( H_2^\prime(n) - H_2^\prime(\bar n) \big) (v - \bar v) + H_2(n | \bar n) \bar v \\
& + (\phi - \bar \phi)(\rho u - nv - \bar \rho \bar u + \bar n \bar v) - \delta (\phi - \bar \phi) \nabla \partial_t (\phi - \bar \phi) ,
\end{split}
\]
and the remaining terms $\hat{\Sigma}_i, \ i = 1,2,*$ are as follows
\begin{equation} \label{RHSterms}
\begin{split}
\hat{\Sigma}_1 & = -\rho \nabla \bar u : (u - \bar u) \otimes (u - \bar u) - \varepsilon n \nabla \bar v :  (v - \bar v) \otimes (v - \bar v)  , \\
\hat{\Sigma}_2 & = -(\nabla \cdot \bar u)P_1(\rho | \bar \rho)  - (\nabla \cdot \bar v)P_2(n | \bar n)   , \\
\hat{\Sigma}_* & = \big( (\rho - \bar \rho) \bar u - (n - \bar n) \bar v   \big) \cdot \nabla (\phi - \bar \phi) .
\end{split}
\end{equation}
The first term above, $\hat{\Sigma}_1$, originates from the contribution of the relative kinetic energy, while the other two terms, $\hat{\Sigma}_2$ and $\hat{\Sigma}_3$, appear from the evolution of the relative potential energy. This computation can be carried out either by direct exact calculation, considering the system of equations satisfied by the difference of the two solutions, or by employing a functional abstract formalism; for details see \cite{alves2022relaxation}. 
\par
The relative energy identity for (\ref{BEP}) is then obtained by integrating the identity (\ref{localformrelenergyidentity}) over space:
\[\frac{d}{dt} \Phi =  \int_\Omega \hat{\Sigma}_1 + \hat{\Sigma}_2 + \hat{\Sigma}_*  \, dx  ,\] 
where $\Phi$ is defined as:
\begin{equation} \label{Phi}
\Phi = \Phi(t) =  \int_\Omega  \hat{\mathcal{H}}(\rho, u , n , v, \phi | \bar \rho,  \bar u, \bar n,  \bar v, \bar \phi) \, dx  ,
\end{equation} 
where $\hat{\mathcal{H}}$ is as (\ref{relativeenergy}).
This function serves as a tool for comparing solutions of (\ref{BEP}) with solutions of (\ref{AE}) and (\ref{E}) in the subsequent sections of the manuscript.
\subsection{Dissipative weak solutions} \label{weaksolutions}
Here, the notion of weak solutions to the equations (\ref{BEP}) is prescribed.  
A tuple $(\rho,u,n,v, \phi)$ such that the densities $\rho,n$ are non-negative and
 \[\rho \in C\big([0,T[;  L^{\gamma_1}(\Omega)\big)  , \qquad  \ n \in C\big([0,T[; L^{\gamma_2}(\Omega)\big)  , \qquad \gamma = \min\{\gamma_1, \gamma_2 \} > 1  ,\]
the momenta $\rho u, nv$ belong to $ C\big([0,T[;L^1(\Omega, \mathbb{R}^d)\big),$
and additionally
  \[\rho |u|^2, \ n|v|^2, \ \rho \nabla \phi, \ n \nabla \phi \in L^1(]0,T[ \times \Omega)  , \]
where $\phi$ solves $-\delta \Delta \phi = \rho - n,$ is a weak solution of (\ref{BEP}) if:
\begin{enumerate}[(i)]
\item $(\rho,u,n,v, \phi)$ satisfies (\ref{BEP}) in the following weak sense:
 \begin{equation*} \label{weak1}
         -\int_0^T \int_\Omega (\partial_t \varphi) \rho \, dxdt-\int_0^T \int_\Omega \nabla \varphi \cdot (\rho u) \, dxdt - \int_\Omega \varphi \rho \big|_{t=0}\, dx=0  ,
        \end{equation*} \\
 \begin{equation*} \label{weak2}
        \begin{split}
         - &  \int_0^T \int_\Omega \partial_t \tilde{\varphi} \cdot (\rho u) \, dxdt -  \int_0^T \int_\Omega \nabla \tilde{\varphi} : \rho u \otimes u \, dx dt \\
        & -  \int_0^T \int_\Omega (\nabla \cdot \tilde{\varphi}) p_1(\rho) \,  dxdt
        - \int_\Omega \tilde{\varphi} \cdot (\rho u)\big|_{t=0} \, dx \\
        = &   -  \int_0^T \int_\Omega \tilde{\varphi}\cdot (\rho \nabla \phi) \, dxdt ,
        \end{split}
        \end{equation*} \\
\begin{equation*} \label{weak3}
         -\int_0^T \int_\Omega (\partial_t \psi) n \, dxdt-\int_0^T \int_\Omega \nabla \psi \cdot (nv) \, dxdt - \int_\Omega \psi n \big|_{t=0} \, dx=0 ,
        \end{equation*} \\
 \begin{equation*} \label{weak4}
        \begin{split}
         - &  \int_0^T \int_\Omega \varepsilon \partial_t \tilde{\psi} \cdot (nv) \, dxdt -  \int_0^T \int_\Omega \varepsilon \nabla \tilde{\psi} :  nv \otimes v \, dx dt \\ & -  \int_0^T \int_\Omega (\nabla \cdot \tilde{\psi}) p_2(n) \, dxdt
         -  \int_\Omega\varepsilon \tilde{\psi} \cdot (nv)\big|_{t=0} \, dx \\
         = &    \int_0^T \int_\Omega \tilde{\psi} \cdot (n \nabla \phi) \, dxdt ,
        \end{split}
        \end{equation*}
for every Lipschitz test functions $\varphi, \psi : \mathopen{[}0,T\mathclose{[} \times \bar{\Omega} \to \mathbb{R}, \ \tilde{\varphi}, \tilde{\psi}:\mathopen{[}0,T\mathclose{[} \times \bar{\Omega} \to \mathbb{R}^d$ compactly supported in time, and such that $\tilde{\varphi} \cdot \nu = \tilde{\psi} \cdot \nu = 0 $ on $\mathopen{[}0,T\mathclose{[} \times \partial \Omega$ for any outer normal vector to the boundary $\nu$,
    \item $(\rho,u,n,v, \phi)$ conserves mass: 
    \begin{equation*} \label{massconservation}
 \int_\Omega \rho \, dx = \int_\Omega n \, dx = M < \infty ,  \quad   0 \leq t < T  ,
  \end{equation*}
  \item  $(\rho,u,n,v, \phi)$ has finite total energy:
  \begin{equation*} \label{energyfinitude}
  \underset{[0,T[}{\text{sup}} \int_\Omega  \mathcal{H}_0(\rho, u , n , v, \phi)  \, dx < \infty ,
 \end{equation*}
 \end{enumerate}
where $\mathcal{H}_0$ is given by (\ref{energyBEP}).
Solutions of (\ref{BEP}) are dependent on $\varepsilon$ and $\delta$, {\it i.e.},  \[(\rho, u,n,v,\phi)=(\rho_{\varepsilon,\delta}, u_{\varepsilon,\delta},n_{\varepsilon,\delta}, v_{\varepsilon,\delta}, \phi_{\varepsilon, \delta}) .\]Nevertheless, we omit this dependence for simplicity. \par  
Moreover, a weak solution $(\rho,u,n,v)$ of (\ref{BEP}) is dissipative if $\sqrt{\rho} u, \sqrt{n} v, \nabla \phi \in C\big([0,T[; L^2(\Omega, \mathbb{R}^d) \big) $ and 
 \begin{equation} \label{weakdissip}
-  \int_0^T   \int_\Omega \mathcal{H}_0(\rho, u , n , v, \phi) \ \dot{\theta} \, dxdt  \leq   \int_\Omega  \mathcal{H}_0(\rho, u , n , v, \phi) \ \theta \,  dx \Big|_{t=0}
  \end{equation} \\
  for all compactly supported and non-negative $\theta$ belonging to $W^{1,\infty}([0,T[)$. \par Fixing $t \in [0,T[$ and choosing $\theta$ above as
\begin{equation} \label{theta}
\theta(s) = 
\begin{dcases}
1 & \quad \text{for} \ 0 \leq s < t, \\
\frac{t-s}{\tau} + 1 & \quad \text{for} \ t \leq s < t + \tau, \\
0 & \quad \text{for} \ t+\tau \leq s < T ,
\end{dcases} 
\end{equation}
for some positive $ \tau$ such that $\tau + t < T,$ after letting $\tau \to 0$ expression (\ref{weakdissip}) becomes
\begin{equation*} \label{weakenergy}
\int_\Omega \mathcal{H}_0(\rho, u , n , v, \phi)  \, dx \Big|_{s=0}^{s=t} \leq 0  ,
\end{equation*}
{\it i.e.}, the energy is a non-increasing function of time.
\section{Zero-electron-mass limit} \label{zem}
In this section, one evaluates the limit of (\ref{BEP}) towards (\ref{AE}) as $\varepsilon \to 0.$ For simplicity, consider $\delta = 1.$ \par  First, one describes the notion of a strong solution for (\ref{AE}). A bounded Lipschitz triplet $(\bar \rho, \bar u, \bar n)$ with positive densities is a strong solution of $(\ref{AE})$ if:
\begin{enumerate}[(i)]
\item $(\bar \rho, \bar u, \bar n)$ satisfies 
\begin{equation*} \label{strongAE}
\begin{split}
\partial_t \bar \rho + \nabla \cdot (\bar \rho \bar u) & = 0, \\
\partial_t \bar u + \bar u \cdot \nabla \bar u & = - \nabla \big( H_1^\prime (\bar \rho)+ \bar H_2^\prime(\bar n) \big), \\
- \Delta H_2^\prime(\bar n) + \bar n  & =  \bar \rho ,
\end{split}
\end{equation*}
for almost every $(t,x) \in \mathopen{]}0,T\mathclose{[} \times \Omega$ and \[\bar u \cdot \nu = \nabla H_2^\prime(\bar n) \cdot \nu  =  0 \ \text{on} \  \mathopen{[}0,T\mathclose{[} \times \partial \Omega  , \]
\item for each $t \in [0,T[$ and each $\kappa > 0$ such that $t + \kappa < T,$ the functions $\varphi, \tilde{\varphi}$ given by \[\varphi = \big(-\tfrac{1}{2}|\bar u|^2 + H_1^\prime(\bar \rho) +  H_2^\prime(\bar n) \big) \theta , \quad \tilde{\varphi} =  \bar u \theta \]
are Lipschitz continuous, where $\theta$ is given by (\ref{theta}).
\end{enumerate}
Regarding the initial conditions, a strong solution $(\bar \rho, \bar u, \bar n)$ is assumed to originate from initial data $(\bar \rho_0, \bar u_0, \bar n_0)$ that satisfy 
\begin{equation*} \label{AEstronginitialbounds1}
\int_\Omega \bar \rho_0 \, dx = \int_\Omega \bar n_0 \, dx = \bar M < \infty  ,
\end{equation*}
\begin{equation} \label{AEstronginitialbounds2}
 \int_\Omega \mathcal{H}_*(\bar \rho_0, \bar u_0, \bar n_0) \, dx < \infty  ,
\end{equation}
where $\mathcal{H}_*$ is as (\ref{energyAE}).
Moreover, for $i,j = 1, \ldots, d,$ the derivatives 
\begin{equation*}
\begin{split}
& \partial_{x_i} \bar \rho  , \quad \partial_t \bar \rho  , \quad \partial^2_{x_i x_j} \bar \rho  , \\ 
& \partial_{x_i} \bar u  , \quad \partial_t \bar u , \quad \partial^2_{x_i x_j} \bar u  , \\
& \partial_{x_i} \bar n  , \quad \partial_t \bar n  , \quad \partial^2_{x_i x_j} \bar n \ , \quad \partial_t \partial_{x_i} \bar n  , \quad \partial^2_t \bar n  , \quad \partial_t \partial^2_{x_i x_j} \bar n  , \quad \partial^2_t \partial_{x_i} \bar n  , 
\end{split}
\end{equation*}
are assumed to be bounded in $\mathopen{[}0,T\mathclose{[} \times \Omega.$ \par 
A strong solution of (\ref{E}) satisfies the energy equation:
\begin{equation*} \label{AEstrongenergypreweakform}
\frac{d}{dt} \int_\Omega \mathcal{H}_*(\bar \rho, \bar u, \bar n) \, dx = 0  .
\end{equation*}
This equation is obtained by integrating expression (\ref{energylocalformAE}), and it implies that (\ref{AEstronginitialbounds2}) holds for any $t \in \mathopen{[}0,T\mathclose{[}.$ \par
Furthermore, one considers strong solutions that are bounded away from vacuum, {\it i.e.},  
\begin{equation} \label{AEboundawayzero}
A_1 \leq \bar \rho \leq B_1  , \quad  A_2 \leq \bar n \leq  B_2  , 
\end{equation}
for some $A_1,A_2 > 0$ and $B_1, B_2 < \infty$.
\par 
In this section, using the relative energy method, one compares a weak solution $(\rho,  u, n, v, \phi)$ of (\ref{BEP}) satisfying (\ref{weakdissip}) with a strong solution $(\bar \rho,  \bar u, \bar n)$ of (\ref{AE}) satisfying (\ref{AEboundawayzero}). In order to do so, one regards the strong solution of (\ref{AE}) as an approximate solution of (\ref{BEP}) by placing 
\begin{equation*} \label{AEapproxquantit}
\begin{split}
  \bar v &= \frac{ \bar \rho \bar u - \nabla \partial_t H_2^\prime(\bar n)}{\bar n} \ , \\ \bar \phi &=  H_2^\prime(\bar n)  .
  \end{split}
\end{equation*}
One easily checks that $ \partial_t \bar n + \nabla \cdot (\bar n \bar v) = 0$ and $\bar n \bar v \cdot \nu = 0$ on $ \mathopen{[}0,T\mathclose{[} \times \partial \Omega$. \par
Thus, the strong solution, henceforth denoted by $(\bar \rho,  \bar u, \bar n,  \bar v, \bar \phi ),$ satisfies the following approximate equations:
\begin{equation} \label{approxAE}
 \begin{split}
  \partial_t \bar \rho + \nabla \cdot (\bar \rho \bar u) & = 0, \\
  \bar \rho \partial_t  \bar u +\bar \rho  (\bar u \cdot \nabla) \bar u & = - \nabla P_1(\bar \rho) - \rho \nabla \bar \phi, \\
  \partial_t \bar n + \nabla \cdot (\bar n \bar v) & = 0, \\
  \varepsilon \big(\bar n \partial_t  \bar v +\bar n  (\bar v \cdot \nabla) \bar v\big) & = - \nabla P_2(\bar n) + \bar n \nabla \bar \phi + \varepsilon \bar e,\\ 
  - \Delta \bar \phi & = \bar \rho - \bar n, 
 \end{split}
\end{equation} 
where the error $\bar e$ is given by
\[\bar e = \bar n \partial_t  \bar v +\bar n  (\bar v \cdot \nabla) \bar v  .\] 
\par 
Recall the relative energy $\Phi$ given by (\ref{Phi}). The main result of this section is a stability estimate for $\Phi$ evaluated in a weak solution of (\ref{BEP}) and a strong solution of (\ref{AE}). The proof will be given in the subsequent two sections of the manuscript.
\begin{theorem} \label{AEmaintheorem}
Assume that $d \geq 2$. Given $(\rho, u,n, v, \phi),$ a weak solution of (\ref{BEP}) satisfying (\ref{weakdissip}) and with $\gamma \geq 2 - \tfrac{1}{d},$ and given
$(\bar{\rho},  \bar u, \bar n,  \bar v, \bar \phi),$ a strong solution of \eqref{AE} satisfying (\ref{AEboundawayzero}), there exists a positive constant $C$ such that 
\begin{equation} \label{AEstabilityestimate}
\Phi(t) \leq e^{CT}\big(\Phi(0) + \varepsilon\big)  \quad  \forall  t \in [0, T [  .
\end{equation}
Thus, as $\varepsilon$ tends to zero, if $\Phi(0) \to 0,$ then
\[ \sup\limits_{t \in [0,T[}\Phi(t) \to 0 . \]
\end{theorem}
\subsection{Relative energy inequality}
The first step in order to establish Theorem \ref{AEmaintheorem} is an inequality satisfied by $\Phi.$ This inequality, called the relative energy inequality, is obtained considering the solutions detailed in Sections \ref{weaksolutions} and \ref{zem}.
\begin{proposition} \label{AErelenergyinequalityprop}
Given $(\rho, u,n,v, \phi)$, a weak solution of (\ref{BEP}) satisfying (\ref{weakdissip}), and given $(\bar{\rho}, \bar{u}, \bar n,  \bar v, \bar \phi),$ a strong solution of (\ref{AE}), the function $\Phi$ satisfies:
\begin{equation} \label{AErelenergyinequality}
\Phi(t)-\Phi(0) \leq \Sigma_1(t) + \Sigma_2(t) + \Sigma_3(t) + \Sigma_*(t)  \quad \ \forall\  t \in [0, T [   ,
\end{equation}
where \[\Sigma_i(t) = \int_0^t \int_\Omega \hat{\Sigma}_i \, dxds\] with $\hat{\Sigma}_i$ as in (\ref{RHSterms}) for $ i = 1,2,*$, and
\[\Sigma_3(t)  = - \int_0^t \int_\Omega  \varepsilon\frac{n}{\bar{n}} \bar{e} \cdot (v - \bar{v}) \, dx ds  . \] 
\end{proposition}
\begin{proof}
This derivation is very similar to the proof of \cite[Proposition 4.4]{alves2022relaxation}, so most of the details are omitted. Proceeding as in \cite{alves2022relaxation}, one obtains that 
\begin{equation*}
 \begin{split}
   \Phi(t)-\Phi(0) \leq  & \ \Sigma_1(t) + \Sigma_2(t) + \Sigma_3(t) \\
   & + \int_0^t \int_\Omega - (\partial_s \bar{\phi})(\rho - n - \bar{\rho} + \bar{n}) +  (\rho \bar{u} - n \bar{v}) \cdot \nabla(\phi - \bar{\phi}) \,  dx ds  .
 \end{split}
\end{equation*}
Now, set $f = \rho  - n$ and $\bar f = \bar \rho - \bar n.$ One wishes to prove that 
\begin{equation} \label{intpartselliptic}
\int_\Omega - (f - \bar f) (\partial_t \bar \phi) \, dx = \int_\Omega - (\phi - \bar \phi) (\partial_t \bar f) \, dx .
\end{equation}
Let $(f_k) \subseteq C_c^\infty(\Omega)$ be such that $f_k \to f$ in $L^\gamma(\Omega)$ as $k \to \infty,$ and let $\phi_k$ be such that $-\Delta \phi_k = f_k.$ Then,
\begin{equation*}
\int_\Omega - (f_k - \bar f) (\partial_t \bar \phi) \, dx = \int_\Omega \Delta(\phi_k - \bar \phi) (\partial_t \bar \phi) \, dx = \int_\Omega - (\phi_k - \bar \phi) (\partial_t \bar f) \, dx  .
\end{equation*}
Letting $k \to \infty$ above yields the desired identity. Indeed,
\begin{equation*}
\Big| \int_\Omega (f_k - \bar f) (\partial_t \bar \phi) \, dx - \int_\Omega (f - \bar f) (\partial_t \bar \phi) \, dx\Big| \leq  \int_\Omega |(f_k - f)(\partial_t \bar \phi)| \, dx  \leq C \|f_k - f \|_{L^\gamma(\Omega)} \to 0  ,
\end{equation*}
and by standard elliptic theory one derives
\begin{equation*}
\begin{split}
\Big| \int_\Omega (\phi_k - \bar \phi) (\partial_t \bar f) \, dx - \int_\Omega (\phi - \bar \phi) (\partial_t \bar f) \, dx\Big|  & \leq C \|\phi_k - \phi \|_{L^\gamma(\Omega)}  \\
& \leq C \|\phi_k - \phi \|_{W^{2,\gamma}(\Omega)} \leq C \|f_k - f \|_{L^\gamma(\Omega)} \to 0 ,
\end{split}
\end{equation*}
Therefore, (\ref{intpartselliptic}) holds and, together with the continuity equations satisfied by the strong solution, implies that 
\[\int_\Omega - (\partial_s \bar{\phi})(\rho -n- \bar{\rho}  + \bar{n}) + (\rho \bar{u} - n \bar{v}) \cdot \nabla(\phi - \bar \phi) \, dx = \int_\Omega \hat{\Sigma}_* \, dx, \]
which concludes the proof.
\end{proof}

\subsection{Bounds for the relative energy} \label{sectionAEbounds}
With the same assumptions as those in Theorem \ref{AEmaintheorem}, it holds that
\begin{equation} \label{AEbounds}
\begin{split}
\Sigma_i(t) & \leq C \int_0^t \Phi(s) \, ds  , \quad i = 1,2,*  , \\
\Sigma_3(t) & \leq C\varepsilon t + C \int_0^t \Phi(s) \, ds  . \\
\end{split}
\end{equation}
Indeed, the first three terms are treated as in \cite[Lemma 4.5, Lemma 4.6, Lemma 4.8]{alves2022relaxation}, respectively. For the remaining term, $\Sigma_*,$ we state a lemma which serves as the counterpart of \cite[Lemma 4.7]{alves2022relaxation} for the present case.
\begin{lemma}
Under the conditions of Theorem \ref{AEmaintheorem}, there exists a positive constant $C$ such that
\begin{enumerate}[(i)]
\item if $\gamma \geq 2$: 
\begin{equation} \
\begin{split}
\int_\Omega \big((\rho - \bar \rho)\bar u - (n - \bar n)\bar v \big) \cdot \nabla(\phi - \bar \phi) \, dx  & \leq C \int_\Omega |\rho - \bar \rho|^2 + |n - \bar n|^2 + |\nabla (\phi - \bar \phi)|^2 \, dx \\
& \leq C \int_\Omega H_1(\rho | \bar \rho) + H_2(n | \bar n) + |\nabla (\phi - \bar \phi)|^2 \, dx.
\end{split}
\end{equation}
\item if $2 - \frac{1}{d} \leq \gamma < 2$:
\begin{equation} \label{ZEMlemma}
\begin{split}
\int_\Omega \big((\rho - \bar \rho)\bar u - (n - \bar n)\bar v \big) \cdot \nabla(\phi - \bar \phi) \, dx  & \leq C \Big(\int_\Omega (|\rho - \bar{\rho}|+|n - \bar{n}|)^q \, dx \Big)^{\frac{2}{q}} \\
& \leq C \int_\Omega H_1(\rho | \bar \rho) + H_2(n | \bar n) \, dx,
\end{split}
\end{equation}
where $q = \frac{2}{3 - \gamma}$.
\end{enumerate}
\end{lemma}
It is clear that the condition in (\ref{AEbounds}) for $\Sigma_\ast$ follows from this lemma. The proof of the lemma is identical to the proof of \cite[Lemma 4.7]{alves2022relaxation}, except for the first inequality of (\ref{ZEMlemma}), which we prove here. Assume that $2 - \frac{1}{d} \leq \gamma < 2$ and let $q= \frac{2}{3-\gamma},$  $q'= \frac{q}{q-1},$ and $q^*=\frac{dq}{d-q}.$ Since $2 - \tfrac{1}{d} \leq \gamma < 2,$ then $1<q < \gamma < 2$, $q <q^\prime \leq q^*$, and the following Sobolev imbedding holds \cite{adams}:
\begin{equation*} 
    W^{1,q}(\Omega) \to L^{q^\prime}(\Omega) .  
\end{equation*}
Combining the above imbedding with H\"{o}lder's inequality yields that
\begin{equation*}
\int_\Omega \big((\rho - \bar \rho)\bar u - (n - \bar n)\bar v \big) \cdot \nabla(\phi - \bar \phi) \, dx \leq C \| |\rho - \bar{\rho}|+|n - \bar{n}| \|_{L^q(\Omega)} \|\nabla(\phi- \bar{\phi}) \|_{W^{1,q}(\Omega)}    .
\end{equation*}
Now, using elliptic $L^q$ regularity theory, one deduces that \[\|\nabla(\phi- \bar{\phi}) \|_{W^{1,q}(\Omega)} \leq \|\phi- \bar{\phi}\|_{W^{2,q}(\Omega)} \leq C \ \|\rho - \bar{\rho}-n+\bar{n} \|_{L^q(\Omega)}  , \]                                                                                                                                             
whence 
\begin{equation*}
  \int_\Omega \big((\rho - \bar \rho)\bar u - (n - \bar n)\bar v \big) \cdot \nabla(\phi - \bar \phi) \, dx \leq C  \| |\rho - \bar{\rho}|+|n - \bar{n}| \|_{L^q(\Omega)}^2 ,
\end{equation*}
as desired. \par
Combining (\ref{AErelenergyinequality}) with (\ref{AEbounds}) results in 
\[ 
\Phi(t) - \Phi(0) \leq C \int_0^t \Phi(s) \, ds + C\varepsilon t  \quad \forall t \in [0,T[ ,
\]
from which (\ref{AEstabilityestimate}) follows by Gronwall's inequality.
\begin{remark}
In \cite{alves2022relaxation} one uses the Neumann function to represent the solution of the Poisson equation, which is not the case in this section. There, the counterpart of the first inequality in (\ref{ZEMlemma}) is obtained combining the properties of the Neumann function with a classical estimate on Riesz potentials. Those ideas will also be used in the last part of the present work; see Section \ref{section_intbyparts}.
\end{remark}
\section{Joint zero-electron-mass and quasi-neutral limits} \label{zemqn} 
In this section, one analyses the limit of (\ref{BEP}) towards (\ref{E}) as $\varepsilon, \delta \to 0.$
First, one describes the notion of strong solution $(\bar \rho, \bar u)$ for the Euler system (\ref{E}). A bounded Lipschitz pair $(\bar \rho, \bar u)$ with $\bar \rho > 0$ is a strong solution of $(\ref{E})$ if:
\begin{enumerate}[(i)]
\item $(\bar \rho, \bar u)$ satisfies 
\begin{equation*} \label{strongE}
\begin{split}
\partial_t \bar \rho + \nabla \cdot (\bar \rho \bar u) & = 0, \\
\partial_t \bar u + \bar u \cdot \nabla \bar u & = - \nabla\big(H_1^\prime(\bar \rho)+H_2^\prime(\bar \rho) \big),
\end{split}
\end{equation*}
for almost every $(t,x) \in \mathopen{]}0,T\mathclose{[} \times \Omega$ and \[\bar u \cdot \nu = 0 \ \text{on} \  \mathopen{[}0,T\mathclose{[} \times \partial \Omega  , \]
\item for each $t \in [0,T[$ and each $\kappa > 0$ such that $t + \kappa < T,$ the functions $\varphi, \tilde{\varphi}$ given by \[\varphi = \big(-\tfrac{1}{2}|\bar u|^2 + H_1^\prime(\bar \rho) +  H_2^\prime(\bar \rho) \big)\theta  , \quad \tilde{\varphi} =  \bar u \theta \]
are Lipschitz continuous, where $\theta$ is given by (\ref{theta}).
\end{enumerate}
Regarding the initial conditions, a strong solution $(\bar \rho, \bar u)$ is assumed to originate from initial data $(\bar \rho_0, \bar u_0)$ satisfying
\begin{equation*} \label{stronginitialbounds1}
\int_\Omega \bar \rho_0 \, dx = \bar M < \infty   ,
\end{equation*}
\begin{equation} \label{stronginitialbounds2}
 \int_\Omega \mathcal{H}_E(\bar \rho_0, \bar u_0) \, dx < \infty  ,
\end{equation}
where $\mathcal{H}_E$ is given by (\ref{energyE}).
Moreover, for $i,j = 1, \ldots, d,$ the derivatives 
\[ \partial_{x_i} \bar \rho  , \quad \partial_{x_i} \bar u  , \quad \partial^2_{x_i x_j} \bar \rho  , \quad \partial^2_{x_i x_j} \bar u   \]
are assumed to be bounded in $\mathopen{[}0,T\mathclose{[} \times \Omega.$ \par 
A strong solution of (\ref{E}) satisfies 
\begin{equation*} \label{strongenergypreweakform}
\frac{d}{dt} \int_\Omega \mathcal{H}_E(\bar \rho, \bar u) \, dx = 0  ,
\end{equation*}
which implies that (\ref{stronginitialbounds2}) holds for any $t \in \mathopen{[}0,T\mathclose{[}.$
Furthermore, one considers strong solutions that are bounded away from vacuum, {\it i.e.}, 
\begin{equation} \label{boundawayzero}
A \leq \bar \rho \leq B   ,
\end{equation}
for some $A > 0$ and $B < \infty.$
\par 
Using the relative energy method, one compares a weak solution $(\rho,  u, n, v, \phi)$ of (\ref{BEP}) satisfying (\ref{weakdissip}) with a strong solution $(\bar \rho, \bar u)$ of (\ref{E}) satisfying (\ref{boundawayzero}). In order to do so, one regards the strong solution as an approximate solution of (\ref{BEP}) by placing 
\begin{equation*} \label{approxquantit}
\begin{split}
\bar n &= \bar \rho ,\\ 
\bar v &= \bar u ,\\
 \bar \phi &=  H_2^\prime(\bar \rho) .
\end{split}
\end{equation*}
Thus, the strong solution, henceforth denoted by $(\bar \rho, \bar u, \bar n, \bar v, \bar \phi ),$ satisfies the following approximate system:
\begin{equation} \label{approxE}
 \begin{split}
  \partial_t \bar \rho + \nabla \cdot (\bar \rho \bar u) & = 0, \\
  \bar \rho \partial_t  \bar u +\bar \rho  (\bar u \cdot \nabla) \bar u & = - \nabla P_1(\bar \rho) - \rho \nabla \bar \phi, \\
  \partial_t \bar n + \nabla \cdot (\bar n \bar v) & = 0, \\
  \varepsilon \big(\bar n \partial_t  \bar v +\bar n  (\bar v \cdot \nabla) \bar v)\big)  & = - \nabla P_2(\bar n) +  \bar n \nabla \bar \phi + \varepsilon \bar e,\\ 
  - \delta \Delta \bar \phi & = \bar \rho - \bar n + \delta \bar e_0,
 \end{split}
\end{equation} 
where the errors $\bar e,$ $\bar e_0$ are given by
\[\bar e = \bar n \partial_t \bar v + \bar n (\bar v \cdot \nabla) \bar v    , \quad \bar e_0 = - \Delta \bar \phi .\]
\par 
Once again, one utilizes the relative energy function $\Phi$ given by (\ref{Phi}) to compare the two considered solutions of the different systems. In this case, one compares a weak solution $(\rho,  u, n, v, \phi)$ of (\ref{BEP}) and a strong solution $(\bar \rho,  \bar u, \bar n, \bar v, \bar \phi )$ of (\ref{E}). Moreover, here one considers the representation formula using the Neumann function $N$ for the solution of $-\delta \Delta \phi = \rho - n$:
\[\phi= N * \Big( \frac{\rho - n}{\delta} \Big)   , \]
see \cite{kenig1994harmonic}. The choice of this representation formula together with the theory of Riesz potentials yield two integration by parts formulas that are crucial to prove the main result of this section. The result, a stability estimate satisfied by $\Phi$, is as follows:
\begin{theorem} \label{maintheorem}
Assume that $d \geq 3$. Given $(\rho, u,n, v),$ $\phi = N*(\rho - n)/\delta$,  a weak solution of (\ref{BEP}) satisfying (\ref{weakdissip}) and with $\gamma \geq \frac{2d}{d+1},$ and given $(\bar{\rho},  \bar u, \bar n,  \bar v, \bar \phi),$ a strong solution of \eqref{E} satisfying (\ref{boundawayzero}), there exists a positive constant $C$ such that 
\begin{equation} \label{stabilityestimate}
\Phi(t) \leq e^{CT}\big(\Phi(0) + \varepsilon + \delta \big) \quad \forall t \in [0, T [  .
\end{equation}
Thus, if $\Phi(0) \to 0$ as $(\varepsilon, \delta) \to (0,0),$ then
\[ \sup\limits_{t \in [0,T[}\Phi(t) \to 0 \  \text{as} \  (\varepsilon, \delta) \to (0,0)  . \]
\end{theorem}
The remainder of the manuscript is devoted to giving a proof of the previous result.
\subsection{Integration by parts formulas} \label{section_intbyparts}

Before stating and proving the integration by parts formulas that will be used in the subsequent analysis, one needs to recall some properties of the Neumann function and the notion of Riesz potentials. These will be crucial to handle the terms containing the electric potential $\phi.$ \par 
The Neumann function $N$ provides a representation formula for the solution of the problem 
\begin{equation*}
\begin{dcases}
 - \Delta \phi = f & \quad \text{in} \ \Omega,\\
 \nabla \phi \cdot \nu = 0 & \quad \text{on} \ \partial \Omega, \ 
\end{dcases}
\end{equation*}
being the counterpart of the Newtonian kernel for the case of a bounded  domain of $\mathbb{R}^d$ with smooth boundary and $d \geq 3$ \cite{kenig1994harmonic}. Its relevant properties which will be used are the following:
\begin{enumerate}[(i)]
 \item $N$ is smooth off the diagonal $\{(x,x) \ | \ x \in \bar{\Omega} \}$,
 \item $N$ is symmetric,
 \item $N(x,z) \leq C |x-z|^{2-d}  ,$
 \item $|\nabla_x N(x,z)| \leq C |x-z|^{1-d}  ,$
\end{enumerate}  
where $C$ is a positive constant. \par 
The Riesz potential of order $\alpha \in \mathopen{]}0,d\mathclose{[}$, $I_\alpha f$, of a function $f : \mathbb{R}^d \to \mathbb{R}$ is given  by
\[I_\alpha f(x) = \int_{\mathbb{R}^d} f(z) |x-z|^{\alpha-d} \ dz  .  \] 
If $f \in L^\gamma(\mathbb{R}^d)$ with $1 < \gamma < d/\alpha,$ then one has \cite{hedberg1972certain}:
\begin{equation} \label{steinprop}
\|I_\alpha f\|_{L^{\frac{d\gamma}{d-\alpha \gamma}}(\mathbb{R}^d)} \leq C \| f \|_{L^\gamma(\mathbb{R}^d)}  ,
\end{equation}
 for some constant $C > 0.$
 

Combining the properties of the Neumann function with estimate (\ref{steinprop}) yields two integration by parts formulas that will be essential in the subsequent analysis. The first formula is proved in \cite[Proposition 4.2]{alves2022relaxation}.
\begin{proposition} \label{propintparts1}
Let $f,g \in L^\gamma(\Omega)$ and $\delta > 0,$ and consider $\phi = N * f/\delta, \ \varphi = N * g/\delta, \ \nabla \phi = \nabla_x N * f/\delta$, and $\nabla \varphi = \nabla_x N * g/\delta$. If $\gamma \geq \frac{2d}{d+2}$, then $\phi, \varphi \in L^{\frac{2d}{d-2}}(\Omega), \ \nabla \phi, \nabla \varphi \in L^{2}(\Omega, \mathbb{R}^d )$, and 
\begin{equation*} 
 \int_\Omega \delta \nabla \phi \cdot \nabla \varphi \, dx = \int_\Omega f \varphi \, dx = \int_\Omega g \phi \, dx = \int_\Omega \int_\Omega f(x)N(x,y)g(y) \, dxdy .
\end{equation*}
In particular, for any $\bar \varphi \in W^{1,\infty}(\Omega)$ it holds that
\begin{equation} \label{intpartsformula1}
\int_{\Omega} \bar \varphi f \, dx = \int_\Omega \delta \nabla \bar \varphi \cdot \nabla \phi \, dx  .  
\end{equation}
\end{proposition}
The next formula is deduced in \cite[Proposition 4.3]{alves2023role}, where $\Omega$ and the Neumann function are substituted by the whole space $\mathbb{R}^d$ and the Newtonian kernel, respectively. The proof in the present case is essentially the same, and is therefore omitted.   
\begin{proposition}
Under the same conditions of Proposition \ref{propintparts1}, if $\gamma \geq \frac{2d}{d+1}$, then $f \nabla \phi \in L^{1}(\Omega,\mathbb{R}^d)$ and for any $\bar u \in W^{1,\infty}(\Omega, \mathbb{R}^d)$ with $\bar u \cdot \nu = 0$ on $\partial \Omega$ it holds that
\begin{equation} \label{intpartsformula2}
\int_\Omega f \nabla \phi \cdot \bar u \, dx = \int_\Omega \delta \nabla \bar u : \nabla \phi \otimes \nabla \phi \, dx - \int_\Omega \delta (\nabla \cdot \bar u) \tfrac{1}{2}|\nabla \phi|^2 \, dx  .
\end{equation}
\end{proposition}
\begin{proof}
We give a proof of the first assertion, {\it i.e.}, $f \nabla \phi \in L^{1}(\Omega,\mathbb{R}^d)$. Set $\gamma_0 = \frac{2d}{d+1}$ and note that $f \in L^{\gamma_0}(\Omega).$
Moreover, it is clear that \[|\nabla \phi(x)| \leq CI_1(|\tilde{f}|)(x)  , \quad x \in \Omega  , \]
where $\tilde{f}$ is the zero extension of $f$ to all of $\mathbb{R}^d.$ Consequently, by (\ref{steinprop}),
\begin{equation*}
\begin{split}
 \|f \nabla \phi\|_{L^1(\Omega)}  & \leq  \|f\|_{L^{\gamma_0}(\Omega)} \|\nabla \phi\|_{L^{\gamma_0^\prime}(\Omega)} \\ 
                                 &  \leq   C \|f\|_{L^{\gamma_0}(\Omega)} \|I_1(|\tilde{f}|)\|_{L^{\gamma_0^\prime}(\Omega)} \\
                                 & \leq   C  \|f\|_{L^{\gamma_0}(\Omega)} \|I_1(|\tilde{f}|)\|_{L^{\frac{d\gamma_0}{d-\gamma_0}}(\mathbb{R}^d)} \\ 
                                 & \leq C  \|f\|_{L^{\gamma_0}(\Omega)} \|\tilde{f}\|_{L^{\gamma_0}(\mathbb{R}^d)} \\
                                & =   C  \| f \|_{L^{\gamma_0}(\Omega)}^2  ,
\end{split}
 \end{equation*}
from which the desired conclusion follows. \par 
\end{proof}

\subsection{Relative energy inequality}
In this section, one presents an inequality satisfied by $\Phi$ for the notions of solutions detailed in Section \ref{weaksolutions} and Section \ref{zemqn}. This inequality is the relative energy inequality for the present case.
\begin{proposition} \label{relenergyinequalityprop}
Given $(\rho, u,n,v),$ $\phi = N*(\rho - n)/\delta,$ a weak solution of (\ref{BEP}) satisfying (\ref{weakdissip}) and with $\gamma \geq \frac{2d}{d+2}$, and given $(\bar{\rho}, \bar{u}, \bar n, \bar v, \bar \phi)$, a strong solution of (\ref{E}), the function $\Phi$ satisfies:
\begin{equation} \label{relenergyinequality}
\Phi(t)-\Phi(0) \leq \Sigma_1(t) + \Sigma_2(t) + \Sigma_3(t) + \Sigma_4(t) + \Sigma_5(t) + \Sigma_6(t) \quad \forall t \in [0,T[ ,
\end{equation}
where $\Sigma_i, \ i=1,2,3$ are as in (\ref{AErelenergyinequality}) and
\begin{equation*}
\begin{split}
\Sigma_4(t) & = - \int_0^t \int_\Omega  \delta (\partial_s \bar \phi) \Delta \bar \phi \,  dx ds  , \\
\Sigma_5(t) & = -  \int_0^t \int_\Omega (\partial_s \bar \phi + \nabla \bar \phi \cdot \bar u)(\rho - n) \,   dx ds  , \\
\Sigma_6(t) & =  \int_0^t \int_\Omega (\rho - n)\nabla \phi \cdot \bar u  \,  dx ds  . 
\end{split}
\end{equation*}
\end{proposition}
\begin{proof}
The inequality is obtained using the same methodology as in the proof of \cite[Proposition 4.4]{alves2022relaxation}, where in this case $\bar \rho = \bar n \ , \bar u = \bar v$ and the energy equality satisfied by the strong solution is 
\begin{equation*} \label{strongenergy}
\int_\Omega \mathcal{H}_E(\bar \rho, \bar u) + \varepsilon \tfrac{1}{2}\bar n|\bar v|^2   - \delta\tfrac{1}{2} |\nabla \bar \phi|^2  \, dx \Big|_{s=0}^{s=t}  = \int_0^t \int_\Omega \varepsilon \bar v \cdot \bar e \, dxds + \int_0^t \int_\Omega \delta (\partial_s \bar \phi) \Delta \bar \phi \, dxds  ,
\end{equation*}
since
\[\int_0^t \int_\Omega \varepsilon \bar v \cdot \bar e \, dxds = \int_\Omega \tfrac{1}{2} \bar n |\bar v|^2 \, dx \Big|_{s=0}^{s=t}  , \] and 
\[-\int_0^t \int_\Omega \delta (\partial_s \bar \phi) \Delta \bar \phi \, dxds = \int_\Omega \delta \tfrac{1}{2} |\nabla \bar \phi|^2 \, dx\Big|_{s=0}^{s=t}  , \] 
where $\mathcal{H}_E$ is as (\ref{energyE}).

\end{proof}

\subsection{Bounds for the relative energy}
Under the same conditions as Theorem \ref{maintheorem} it holds that:
\begin{equation} \label{bounds}
\begin{split}
\Sigma_i(t) & \leq C \int_0^t \Phi(s) \, ds , \quad i = 1,2 , \\
\Sigma_3(t) & \leq C\varepsilon t + C \int_0^t \Phi(s) \, ds  , \\
\Sigma_4(t) & \leq C\delta t  , \\
\Sigma_j(t) & \leq C\delta t + C \int_0^t \Phi(s) \, ds  , \quad j = 5,6  .
\end{split}
\end{equation}
The first three terms are treated in the same fashion as in Section \ref{sectionAEbounds}.
Moreover,
\begin{equation*}
\Sigma_4(t) \leq \|(\partial_s \bar \phi) \Delta \bar \phi \|_\infty \int_0^t \int_\Omega \delta \, dxds  \leq C\delta t  .
\end{equation*}

Using the integration by parts formula (\ref{intpartsformula1}) one obtains
\begin{equation*}
\begin{split}
\Sigma_5(t) & = - \int_0^t \int_\Omega \delta \nabla(\partial_s \bar \phi + \nabla \bar \phi \cdot \bar u)\cdot \nabla \phi \,   dx ds \\
& = - \int_0^t \int_\Omega \delta \nabla(\partial_s \bar \phi + \nabla \bar \phi \cdot \bar u)\cdot \nabla (\phi - \bar \phi) \,   dx ds - \int_0^t \int_\Omega \delta \nabla(\partial_s \bar \phi+ \nabla \bar \phi \cdot \bar u)\cdot \nabla \bar \phi \,   dx ds \\
& \leq \int_0^t \int_\Omega \delta \tfrac{1}{2}|\nabla(\partial_s \bar \phi+ \nabla \bar \phi \cdot \bar u)  |^2 + \delta \tfrac{1}{2}|\nabla(\phi - \bar \phi)|^2  \, dx ds + \|\nabla(\partial_s \bar \phi+ \nabla \bar \phi \cdot \bar u)\cdot \nabla \bar \phi \|_\infty \int_0^t \int_\Omega \delta \, dx ds \\
& \leq C\delta t + C \int_0^t \Phi(s) \, ds   ,
\end{split}
\end{equation*}
whereas the integration by parts formula (\ref{intpartsformula2}) yields
\begin{equation*}
\begin{split}
\Sigma_6(t) &  = \int_0^t \int_\Omega \delta \nabla \bar u : \nabla \phi \otimes \nabla \phi  \, dx ds - \int_0^t \int_\Omega \delta (\nabla \cdot \bar u) \tfrac{1}{2}|\nabla \phi|^2 \, dx ds \\
& \leq C \int_0^t \int_\Omega \delta |\nabla \phi|^2 \, dx ds \\
& =  C \int_0^t \int_\Omega \delta |\nabla (\phi - \bar \phi) + \nabla \bar \phi|^2 \, dx ds \\
& = C \int_0^t \int_\Omega  \delta |\nabla(\phi - \bar \phi)|^2 + \delta \nabla(\phi - \bar \phi) \cdot \nabla \bar \phi + \delta |\nabla \bar \phi|^2   \, dxds \\
& \leq C \delta t + C\int_0^t \Phi(s) \,  ds  .
\end{split}
\end{equation*}
\par
Combining (\ref{relenergyinequality}) with (\ref{bounds}) results in 
\[ 
\Phi(t) - \Phi(0) \leq C \int_0^t \Phi(s) \ ds + C(\varepsilon+\delta)t , \quad t \in [0,T[ ,
\]
from which (\ref{stabilityestimate}) follows by Gronwall's inequality. 

\subsection*{Acknowledgments}
NJA would like to thank Rog\'{e}rio Jorge for many useful conversations regarding the physics of the problem. This work was completed while NJA was a PhD student at King Abdullah University of Science and Technology (KAUST).


\begin{thebibliography}{}

\bibitem{adams} Adams, R. A., \& Fournier, J. J.: Sobolev spaces. Elsevier (2003)

\bibitem{alves2023role} Alves, N. J.: The role of Riesz potentials in the weak-strong uniqueness for Euler-Poisson systems. Applicable Analysis (2023)

\bibitem{alves2022relaxation} Alves, N. J., \& Tzavaras, A. E.: The relaxation limit of bipolar fluid models. Discrete \& Continuous Dynamical Systems, 42(1), 211-237 (2022)

\bibitem{carrillo2020relative} Carrillo, J. A., Peng, Y., \& Wr\'{o}blewska-Kami\'{n}ska, A.: Relative entropy method for the relaxation limit of hydrodynamic models. Networks and Heterogeneous Media, 15(3), 369-387 (2020)

\bibitem{chen1984introduction}
\newblock Chen, F. F.:
\newblock Introduction to Plasma Physics and Controlled Fusion,
\newblock New York: Plenum press, \textbf{1} (1984)


\bibitem{freidberg2008plasma} Freidberg, J. P.: Plasma physics and fusion energy. Cambridge university press (2008)

\bibitem{guo2011global} Guo, Y., \& Pausader, B.: Global smooth ion dynamics in the Euler-Poisson system. Communications in Mathematical Physics, 303, 89-125 (2011)

\bibitem{guo2014global} Guo, Y., Ionescu, A. D., \& Pausader, B.: Global solutions of certain plasma fluid models in three-dimension. Journal of Mathematical Physics, 55(12), 123102 (2014)

\bibitem{guo2016global} Guo, Y., Ionescu, A. D., \& Pausader, B.: Global solutions of the Euler-Maxwell two-fluid system in 3D. Annals of Mathematics, 377-498 (2016)

\bibitem{hedberg1972certain} Hedberg, L. I.: On certain convolution inequalities. Proceedings of the American Mathematical Society, 36(2), 505-510 (1972)

\bibitem{horton1999drift} Horton, W.: Drift waves and transport. Reviews of Modern Physics, 71(3), 735 (1999)

\bibitem{hou2017global} Hou, F., \& Yin, H.: On the global existence and blowup of smooth solutions to the multi-dimensional compressible Euler equations with time-depending damping. Nonlinearity, 30(6), 2485 (2017)

\bibitem{jorge2018theory} Jorge, R., Ricci, P., \& Loureiro, N. F.: Theory of the drift-wave instability at arbitrary collisionality. Physical review letters, 121(16), 165001 (2018).

\bibitem{ju2019quasineutral}
Ju, Q., \& Li, Y.: Quasineutral limit of the two-fluid Euler-Poisson system in a bounded domain of $\mathbb{R}^3$. Journal of Mathematical Analysis and Applications, 469(1), 169-187 (2019).

\bibitem{kenig1994harmonic} Kenig C. E.: Harmonic Analysis Techniques for Second Order Elliptic Boundary Value Problems, Vol. 83, American Mathematical Society (1994)

\bibitem{lattanzio2000bipolar} Lattanzio, C.: On the 3-D bipolar isentropic Euler-Poisson model for semiconductors and the drift-diffusion limit. Mathematical Models and Methods in Applied Sciences, 10(03), 351-360 (2000)

\bibitem{lattanzio1999relaxation} Lattanzio, C., \& Marcati, P.: The relaxation to the drift-diffusion system for the 3-$ D $ isentropic Euler-Poisson model for semiconductors. Discrete \& Continuous Dynamical Systems-A, 5(2), 449-455 (1999)

\bibitem{lattanzio2017gas}Lattanzio, C., \& Tzavaras, A. E.: From gas dynamics with large friction to gradient flows describing diffusion theories. Communications in Partial Differential Equations, 42(2), 261-290 (2017)


\bibitem{lattanzio2013relative} Lattanzio, C., \& Tzavaras, A. E.: Relative entropy in diffusive relaxation. SIAM Journal on Mathematical Analysis, 45(3), 1563-1584 (2013)



\bibitem{loeper2005quasi}
Loeper, G.: Quasi-neutral limit of the Euler-Poisson and Euler-Monge-Ampere systems. Communications in Partial Differential Equations, 30(8), 1141-1167 (2005)

\bibitem{markowich2012semi} Markowich, P. A., Ringhofer, C. A., \& Schmeiser, C.: Semiconductor equations. Springer Science \& Business Media (2012)

\bibitem{peng2022global}
Peng, Y. J., \& Liu, C.: Global quasi-neutral limit for a two-fluid Euler-Poisson system in one space dimension. Journal of Differential Equations, 330, 81-109 (2022)


\bibitem{schochet1986compressible} Schochet, S.: The compressible Euler equations in a bounded domain: existence of solutions and the incompressible limit. Communications in Mathematical Physics, 104(1), 49-75 (1986)

\bibitem{tzavaras2005relative} Tzavaras, A. E.: Relative entropy in hyperbolic relaxation. Communications in Mathematical Sciences, 3(2), 119-132 (2005)

\bibitem{xu2013zero} Xu, J., \& Zhang, T.: Zero-electron-mass limit of Euler-Poisson equations. Discrete and Continuous Dynamical Systems, 33(10), 4743-4768 (2013)

\end{thebibliography}
\end{document}